\documentclass[12pt]{amsart}
\usepackage{amsthm}
\newtheorem{theorem}{Theorem}[section]
\newtheorem{lemma}[theorem]{Lemma}
\newtheorem{proposition}[theorem]{Proposition}
\theoremstyle{definition}

\theoremstyle{remark}
\newtheorem{example}[theorem]{Example}
\newtheorem{remark}[theorem]{Remark}
\newtheorem{corollary}[theorem]{Corollary}

\counterwithin*{section}{part}

\usepackage{mathtools}
\usepackage{relsize}

\usepackage{amssymb}
\usepackage{empheq}
\usepackage{stmaryrd}
\usepackage{enumerate}
\usepackage[shortlabels]{enumitem}
\usepackage{calc}
\usepackage{url}

\usepackage{mathabx}

\usepackage{xcolor}

\usepackage{mathtools}
\DeclarePairedDelimiter{\ceil}{\lceil}{\rceil}
\DeclarePairedDelimiter{\floor}{\lfloor}{\rfloor}

\newcommand{\seq}[1]{\{{#1}\}}

\def\EE{\mathbb{E}}
\def\PP{\mathbb{P}}
\def\NN{\mathbb{N}}
\def\ZZ{\mathbb{Z}}

\def\RR{\mathbb{R}}

\usepackage[left=2.0cm,%
                right=2.0cm,%
                top=2.5cm,
                bottom=3.5cm,%
                headheight=12pt,%
                a4paper]{geometry}%

\begin{document}
\title{Quantitative Strong Laws of Large Numbers}

\author{Morenikeji Neri}
\maketitle
\vspace*{-5mm}
\begin{center}
{\scriptsize 
Department of Computer Science, University of Bath\\
E-mail: mn728@bath.ac.uk}
\end{center}
\maketitle
\begin{abstract}
Using proof-theoretic methods in the style of proof mining, we give novel computationally effective limit theorems for the convergence of the Cesaro-means of certain sequences of random variables. These results are intimately related to various Strong Laws of Large Numbers and, in that way, allow for the extraction of quantitative versions of many of these results. In particular, we produce optimal polynomial bounds in the case of pairwise independent random variables with uniformly bounded variance, improving on known results; furthermore, we obtain a new Baum-Katz type result for this class of random variables. Lastly, we are able to provide a fully quantitative version of a recent result of Chen and Sung that encompasses many limit theorems in the Strong Laws of Large Numbers literature.
\end{abstract}
\section{Introduction}
\label{sec:intro}
Throughout this article, fix a probability space $(\Omega, \mathbf{F}, \PP)$ which all the random variables we work with will act on.

The classical Strong Law of Large Numbers is the following fundamental result due to Kolmogorov: 
\begin{theorem}[The classical Strong Law of Large Numbers, c.f.\ Theorem 6.6.1 of \cite{gut2005probability}]\label{thrm:SLLN} 
Suppose $X_1, X_2,\ldots$ are independent, identically distributed (iid) real-valued random variables with  $\EE(|X_1|) < \infty$. Then,
\begin{equation}
\label{eqn:SLLN}
       \frac{1}{n}\sum_{i=1}^nX_i\to \EE(X_1)
\end{equation}
 almost surely, that is, with probability $1$.
\end{theorem}
For ease write $S_n := \sum_{i=1}^nX_i$ and assume $\EE(X_1) = 0$. By multiple applications of the continuity of the probability measure, one can show (following the nation in \cite{siegmund1975large}) that the conclusion of the Strong Law of Large Numbers is equivalent to the sequence of real numbers 
\begin{equation*}
    P^*_{n,\varepsilon} := \PP\left(\sup_{m \ge n} \left|\frac{S_m}{m}\right| > \varepsilon\right)
\end{equation*}
converging to $0$ as $n \to \infty$\footnote{The proof of equivalence is similar to how one proves Egorov's theorem and is well known in the literature.}, for any $\varepsilon > 0$, and so the computational content of the Strong Law of Large Numbers can be given an explicit meaning by studying the convergence speed of this sequence of real numbers. That is, for each $\varepsilon > 0$ can we find a function $r_\varepsilon: (0,\infty)\to (0,\infty)$ such that 
\begin{equation*}
    \forall \lambda >0\,\forall n \ge r_\varepsilon(\lambda)\, P^*_{n,\varepsilon} \le \lambda?
\end{equation*}
Such a function is known as a rate of convergence, more generally a rate of convergence for a sequence of real numbers $\seq{x_n}$ converging to some real number $x$ is any function $r: (0,\infty) \to (0,\infty)$ satisfying
\begin{equation*}
    \forall \varepsilon > 0 \, \forall n \ge r(\varepsilon)\left( |x_n - x| \le \varepsilon\right).
\end{equation*}
A strictly decreasing rate of convergence, $r$, that has an inverse can easily be converted to an asymptotic upper bound as follows, 
    \begin{equation*}
        \forall n \in \NN \left( |x_n - x| \le r^{-1}(n)\right).
    \end{equation*}
Furthermore, it is also clear that a function bounding a rate of convergence will also be a rate of convergence. 

A particular instance of a rate of convergence that will occur a lot in this article is that for the partial sums of a convergent series, $\sum_{i=1}^\infty x_i < \infty$, of nonnegative real numbers $\seq{x_i}$. In this case, a rate of convergence for the partial sums, $\sum_{i=1}^nx_i$, to their limit, $\sum_{i=1}^\infty x_i$, will be a function $r: (0,\infty) \to (0,\infty)$ satisfying, for all $\varepsilon >0$ and $n \ge r(\varepsilon)$, 
\begin{equation*}
    \sum_{i = n+1}^\infty x_i \le \varepsilon.
\end{equation*}
It has been of interest to study sufficient conditions for $\frac{S_n}{n}\to 0$ almost surely when we weaken the iid condition. In \cite{etemadi1981elementary}, Etemadi demonstrates that $\frac{S_n}{n}\to 0$ almost surely if we only assume the random variables are pairwise iid. Furthermore, Etemadi's proof is rather elementary compared to Kolmogorov's original proof of this result for iid random variables.  

To obtain $\frac{S_n}{n}\to 0$, almost surely, dropping the assumption that the random variables are identically distributed requires that other assumptions are placed on the random variables instead. If we keep the independence condition, one can obtain another classical result of Kolmogorov:
\begin{theorem}[Kolmogorov's Strong Law of Large Numbers]
\label{thrm:kolSL}
Suppose $\seq{X_n}$ is a sequence of independent random variables, each with expected value 0 and 
\begin{equation}
 \label{eqn:varsumfin}
        \sum_{n = 1}^\infty \frac{\mathrm{Var}(X_n)}{n^2} < \infty.
\end{equation}
Then $\frac{S_n}{n}\to 0$ almost surely.
\end{theorem}
In \cite{csorgHo1983strong} Cs\"org\H o et al.\ demonstrate that Kolmogorov's Strong Law of Large Numbers does not hold if we weaken the independence condition, in Theorem \ref{thrm:kolSL}, to even pairwise independence (see  \cite[Theorem 3]{csorgHo1983strong}). They instead obtain the following:
\begin{theorem}[cf.\ Theorem 1 of \cite{csorgHo1983strong}]\label{thrm:csorgo}
Suppose $\seq{X_n}$ is a sequence of pairwise independent random variables, each with expected value 0, satisfying (\ref{eqn:varsumfin}) and 
  \begin{equation}
  \label{eqn:bigoh}
        \frac{1}{n}\sum_{k = 1}^n \EE(|X_k|) = O(1).
    \end{equation}
Then $\frac{S_n}{n}\to 0$ almost surely.
\end{theorem}
This article will study the quantitative content of the Strong Law of Large Numbers when the iid assumption is weakened by calculating explicit rates of convergences for $P^*_{n, \varepsilon}$. We will do this by proving an abstract technical theorem (in Section \ref{sec:gen}) whose quantitative content captures the key combinatorial idea in the proof of Theorem \ref{thrm:csorgo}. Our abstract theorem will allow us to produce our first main contribution:
\begin{theorem}
\label{thrm:bound}
Suppose $\seq{X_n}$ is a sequence of pairwise independent random variables satisfying, $\EE(X_n) = 0$, $\EE(|X_n|) \le\tau$ and $\mathrm{Var}(X_n) \le \sigma^2$,  for all $n \in \NN$ and some $\tau, \sigma >0$. There exists a universal constant $\kappa \le 1536$ such that for all $0< \varepsilon \le \tau$,
  \begin{equation*}
   P^*_{n,\varepsilon} \le \frac{\kappa\sigma^2\tau}{n\varepsilon^3}. 
    \end{equation*}
\end{theorem}
The above is an improvement of the best known asymptotic upper bound, given by Luzia \cite{luzia_2018}, who showed (with notation as in Theorem \ref{thrm:bound}) that for all $\beta > 1$ and $0<\varepsilon\le \tau$, there exists $N(\beta, \varepsilon, \tau)$ such that, for all $n \ge N(\beta,\varepsilon,\tau)$,
\begin{equation*}
   P^*_{n,\varepsilon} \le \frac{\sigma^2}{n\varepsilon^2}(C_\beta+D_\beta\log(n)^{\beta - 1}),
\end{equation*}
for some $C_\beta, D_\beta >0$ depending only on $\beta$. Thus, if we fix $\varepsilon, \sigma$, and $\tau$ the above tells us that for each $\beta > 1$,
\begin{equation*}
   P^*_{n,\varepsilon} = O\left(\frac{\log(n)^{\beta - 1}}{n}\right)
\end{equation*}
 as $n \to \infty$, for such a class of random variables, whereas the bound in Theorem \ref{thrm:bound} yields:
 \begin{equation*}
   P^*_{n,\varepsilon} = O\left(\frac{1}{n}\right).
\end{equation*}
 Observe we can take $\tau = \sigma$ by Jensen's inequality, so Theorem \ref{thrm:bound} in particular yields
\begin{equation*}
   P^*_{n,\varepsilon} \le \frac{\kappa\sigma^3}{n\varepsilon^3}.
 \end{equation*}
The above bound bears a resemblance to the bound one obtains in the case where the random variables are assumed to be independent, namely, the H{\'a}jek and R{\'e}nyi \cite{hajek1955generalization} states that if $\seq{X_n}$ is a sequence of independent random variables each with expected value $0$ then, for all $\varepsilon > 0$ and $n < m$
\begin{equation*}
    \PP\left(\max_{n\le k \le m}\left|\frac{S_k}{k}\right|> \varepsilon\right) \le \frac{1}{\varepsilon^2}\left(\frac{1}{n^2}\sum_{k=1}^n\mathrm{Var}(X_k) + \sum_{k=n+1}^m \frac{\mathrm{Var}(X_k)}{k^2}\right).
\end{equation*}
 From this, one easily obtains that if  $\seq{X_n}$ is also assumed to have a common bound on their variance $\sigma^2$,
\begin{equation}
\label{eqn:iidHR}
    P^*_{n,\varepsilon} \le \frac{2\sigma^2}{n\varepsilon^2}.
\end{equation}
However, the H{\'a}jek and R{\'e}nyi inequality generalises Kolmogorov's inequality, which is known to fail for pairwise independent random variables; therefore, more work is needed to obtain a bound in this case. 

In addition, through a simple construction, we demonstrate that $O(1/n)$ is the best polynomial bound for $P^*_{n,\varepsilon}$ in the case that $\seq{X_n}$ is iid with finite variance. That is, for each $\delta > 0$, we construct a sequence of random variables with finite variance, satisfying 
 \begin{equation*}
    \PP\left(\sup_{m \ge n} \left|\frac{S_{m}}{m}\right| > \varepsilon\right) \ge \frac{\omega}{n^{1+\delta}}
     \end{equation*}
for some $\omega > 0$ depending only on $\delta$, for all $0< \varepsilon \le 1$. All of these results are in Section \ref{sec:finvar}.\\

The following result is mainly attributed to Baum and Katz \cite{baum1965convergence} as well as Chow \cite{chow1973delayed}:
\begin{theorem}
\label{thrm:baum-katzfull}
    Let $\seq{X_n}$ be a sequence of iid random variables satisfying $\EE(X_1) = 0$ and let $r \ge -1$. Then for all $\varepsilon > 0$, the following are equivalent:
    \begin{itemize}
        \item[(i)] $\EE(|X_1|^{r+2}) < \infty$,
        \item [(ii)]  $\sum_{n = 1}^\infty n^r\PP\left(\left|\frac{1}{n}S_n\right|>\varepsilon\right) < \infty$,
        \item [(iii)]  $\sum_{n = 1}^\infty n^r\PP\left(\sup_{m \ge n}\left|\frac{1}{m}S_m\right|>\varepsilon\right) < \infty$,
        \item [(iv)] $\sum_{n = 1}^\infty n^r\PP\left(\max_{1\le m \le n}|S_m|>n\varepsilon\right) < \infty$.
    \end{itemize}
\end{theorem}
To prove this result, independence is crucial. Work has been done to extend this result to the case where the random variables are pairwise independent. It is clear that $(iii)$ and $(iv)$ both imply $(ii)$ in the non-independent case. However, as noted in \cite{bai2014complete}, it is possible that $(iv)$ is strictly stronger than $(ii)$ in the non-independent case, and work has been done in establishing the convergence of $(iv)$ in the case where the random variables are pairwise iid. Many authors have established the following theorem, but the result goes back to Rio \cite{rio1995speeds}.
\begin{theorem}
    Suppose $\seq{X_n}$ are pairwise independent, identically distributed random variables with $\EE(X_1) = 0$. For all $-1 \le r < 0$: $\EE(|X_1|^{2+r}) < \infty$ iff
        \begin{equation*}
        \sum_{n = 1}^\infty n^r\PP\left(\max_{1\le m \le n}|S_m|>n\varepsilon\right) < \infty
    \end{equation*}
for all $\varepsilon > 0$.
\end{theorem}
There does not appear to be any results in the literature for the convergence of the sum $(iii)$, assuming the random variables are pairwise independent. However, a simple application of Theorem \ref{thrm:bound} gives the following:
\begin{corollary}
\label{cor:baum-katznew}
    Suppose $\seq{X_n}$ are pairwise independent, identically distributed random variables with $\EE(X_1) = 0$ and $\mathrm{Var}(X_1) < \infty$. Then, for all $\varepsilon > 0$ and $r<0$:
    \begin{equation*}
        \sum_{n = 1}^\infty n^rP^*_{n,\varepsilon} < \infty.
    \end{equation*}
\end{corollary}
\begin{proof}
    This result simply follows from the fact that $ P^*_{n,\varepsilon} = O\left(\frac{1}{n}\right)$.
\end{proof}
Furthermore, it appears to be open whether it is the case that condition $(iii)$ in Theorem \ref{thrm:baum-katzfull} holds in the case $r = 0$ and if the random variables are only assumed to be pairwise independent, which is the case for iid random variables, by Theorem \ref{thrm:baum-katzfull}.\\

In \cite{etemadi1981elementary}, Etemaidi's novel insight in demonstrating that $\frac{S_n}{n}\to 0$ almost surely for pairwise iid random variables was that one could first assume that the random variables were nonnegative, in which case one can take advantage of the monotonicity of the partial sums. The general case is then obtained by using the decomposition of a random variable into its positive and negative parts (that is, writing a random variable, $X$, as $X = X^+ - X^-$ where $X^+ = \max\{X, 0\}$ and $X^- = \max\{-X, 0\}$). Due to this insight, there has been a lot of interest in studying when $\frac{S_n}{n}\to 0$, almost surely, for nonnegative random variables that are not assumed to be iid, as e.g.\ in Petrov \cite{petrov2009}, which was later generalised by Korchevsky et al.\ \cite{korchevsky2010strong} and further generalised again by Korchevsky in \cite{korchevsky2015generalization}. In addition, Chandra et al.\ \cite{chandra1992cesaro} generalise Theorem \ref{thrm:csorgo}, with Chen and Sung \cite{CHEN201680} later producing a result which unifies \cite{chandra1992cesaro} and \cite{korchevsky2015generalization}, as well as generalising results from \cite{birkel1988note,kim1999strong,nili2004strong,jabbari2013almost}. The proofs of all the results Chen and Sung generalise are adaptations of the proof of Theorem \ref{thrm:csorgo}, and they established the following general sufficient condition, which encompasses all the results mentioned:
\begin{theorem}[cf.\ Theorem 2.1 of \cite{CHEN201680}]\label{thrm:chen}
Let $\seq{X_n}$ be a sequence of nonnegative random variables with finite $p$-th moment (for some fixed $p\ge 1$) and respective expected values $\seq{\mu_n}$. Let $S_n := \sum_{k=1}^n X_k$ and $z_n := \sum_{i = 1}^n\mu_i.$ Suppose that
    \begin{equation*}
        \frac{z_n}{n} = O(1)
    \end{equation*}
     and that there exists a sequence of nonnegative real numbers $\seq{\gamma_n}$ satisfying
         \begin{itemize}
             \item $\EE(|S_n - z_n|^p) \le \sum_{k=1}^n \gamma_k$,
             \item $\sum_{n=1}^\infty \frac{\gamma_n}{n^p} < \infty$.
         \end{itemize}
         Then
         \[
         \frac{S_n}{n} - \frac{z_n}{n}\to 0
         \]
         almost surely.
\end{theorem}
Building on our first main contribution, which was an improvement of the bound in \cite{luzia_2018}, our second main contribution will be a fully quantitative version of Theorem \ref{thrm:chen}.
\begin{theorem}\label{thrm:quantchen}
Let $\seq{X_n}$ be a sequence of nonnegative random variables with finite $p$-th moment (for some fixed $p\ge 1$) and respective expected values $\seq{\mu_n}$. Let $S_n := \sum_{k=1}^n X_k$ and $z_n := \sum_{i = 1}^n\mu_i$. Suppose there exists a sequence of nonnegative real numbers $\seq{\gamma_n}$ satisfying
$$\EE(|S_n - z_n|^p) \le \sum_{k=1}^n \gamma_k$$
and
$$\sum_{m = 1}^\infty \frac{\gamma_m}{m^p} \le \Gamma$$
for some $\Gamma \ge 1$, with the partial sums of the above series converging to their limit with a strictly decreasing rate of convergence $\Psi$. Furthermore, assume for all $n \in \NN$, 
 $$\frac{z_n}{n} \le W,$$ 
for some $W \ge 1$. Then for all $0< \varepsilon \le 1, \lambda> 0$ and all
    $$ n \ge A_p\left(\frac{W\Gamma}{\lambda\varepsilon^{p+1}}\right)^{\frac{1}{p}}\Psi\left(\frac{B_p\lambda\varepsilon^{p+1}}{W}\right),$$
it holds that
        \begin{equation*}
       \PP\left(\sup_{m \ge n} \left|\frac{S_m}{m} -\frac{z_m}{m}\right| > \varepsilon\right)\le \lambda.
    \end{equation*}
   Here, $A_p$ and $B_p$ are constants that only depend on p.
\end{theorem}
 From Theorem \ref{thrm:quantchen}, one can obtain quantitative versions of many of the ``Strong Law of Large Numbers''-type results discussed above. For example, we can easily obtain a quantitative version of Theorem \ref{thrm:csorgo}:
\begin{theorem}\label{thrm:quant:cs}
    Suppose $\seq{X_n}$ is a sequence of pairwise independent random variables, each with expected value $0$ and finite variance. Let $S_n := \sum_{k=1}^n X_k$ and $z_n := \sum_{i = 1}^n\EE(|X_i|)$. Further, assume
    $$\sum_{n = 1}^\infty \frac{\mathrm{Var}(X_n)}{n^2} \le \Gamma$$
 for some $\Gamma \ge 1$ and that the partial sums of the above series converge to their limit with a strictly decreasing rate of convergence $\Psi$. Furthermore, assume for all $n \in \NN$, 
 $$\frac{z_n}{n} \le W,$$ 
 for some $W \ge 1$.
    For all $0<\varepsilon \le 1, \lambda> 0$ and all 
    $$ n \ge A\left(\frac{W\Gamma}{\lambda\varepsilon^3}\right)^{\frac{1}{2}}\Psi\left(\frac{B\lambda\varepsilon^{3}}{W}\right),$$
    it holds that
        \begin{equation*}
         P^*_{n,\varepsilon} = \PP\left(\sup_{m \ge n} \left|\frac{S_m}{m}\right| > \varepsilon\right)\le \lambda.
    \end{equation*}
   Here, $A$ and $B$ are universal constants.
\end{theorem}
All of these results are in Section \ref{sec:chen-sung}.\\

This article can be seen as a contribution to the
proof mining program, which aims to use tools and ideas from logic to extract quantitative data and generalisations (through removing superfluous assumptions, for example) from
proofs in mathematics, as well as explain certain computational phenomena one observes when extracting data from proofs, whether that is the independence of a bound from certain parameters of the stated result or the additional quantitative assumptions needed for the extraction of a bound. For example, observe that all of the quantitative results we give in this section enjoy a lot of uniformity. That is, our rates are independent of the distribution of the random variables, the underlying probability space and measure. Furthermore, notice the roles $\Gamma$ and $W$ play in the two previously stated theorems. The rates are not dependent on exact limits; instead, the bounds for such limits suffice. Such uniformity is explained by the underlying logical methods coming from proof mining, which will be briefly discussed in Section \ref{subsec:proof mining}.
\subsection{Related work}
\label{subsec:related}
The study of large deviations in the Strong Law of Large Numbers starts with Cram\'{e}r's 1938 article \cite{cramer1938new}, where he determined large deviation probabilities for the sums of iid random variables up to asymptotic equivalence. Furthermore, in this work, he introduced the moment generating function condition (the moment generating function of the random variables is finite on an interval), which has become a standard assumption in this area.

The subsequent notable work in this direction was in 1960 by Bahadur and Ranga Rao \cite{bahadur1960deviations}, where they built on Cram\'{e}r's work to calculate large deviation probabilities for the weak law of large numbers up to asymptotic equivalence (again assuming the moment generating function condition from Cram\'{e}r).

Then, in 1975, Siegmund \cite{siegmund1975large} (see also \cite{Fill:83:Largedev}) was able to determine $P^*_{n,\varepsilon}$ up to asymptotic equivalence, again assuming the moment generating function condition. Thus, \cite{siegmund1975large} provides the first quantitative interpretation of the Strong Law of Large Numbers. Furthermore, Siegmund's bounds heavily depend on the distribution of the random variables, so although their results are much stronger than those in this article, they assume a lot more about the sequence of random variables.

Not much work has been done to study the large deviations without strong conditions, such as the moment-generating function condition. This may be because, for weaker conditions, one cannot hope to calculate these probabilities up to asymptotic equivalence. The best we can hope for are bounds on the large deviation probabilities. As discussed already, in 2018, Luzia \cite{luzia_2018} obtained distribution independent bounds under milder assumptions on the random variables, which are improved in this article.

Work has been done to study the large deviation probabilities for sequences of random variables that are not necessarily identically distributed. In 1943, Feller was able to generalise  Cram\'{e}r's 1938 article to random variables that are not necessarily identically distributed; however, his assumptions were too restrictive (he assumed the random variables only took values in finite intervals) that the result was not a complete generalisation of Cram\'{e}r's. Petrov \cite{petrov1954generalization}, in 1954, was able to provide a full generalisation of Cram\'{e}r's result and has been able to strengthen this result (by relaxing the moment generating function condition) a further two times, with the most recent in 2006 \cite{petrov122006large} jointly with Robinson.

We also note that Pointwise Ergodic Theorem can be used to show that the Strong Law of Large Numbers also holds for stationary sequences of random variables and obtaining rates for $P^*_{n,\varepsilon}$, in this case, has been of great interest. For example, Gaposhkin \cite{gaposhkin2002some} provides an asymptotic upper bound for $P^*_{n,\varepsilon}$ (which they demonstrate is optimal) for second-order stationary sequences of random variables with finite variance, with more recent work being done by Kachurovskii on this topic, see \cite{kachurovskii2021maximum,kachurovskii2019measuring}.

 Baum-Katz type rates have been obtained in the Strong Law of Large Numbers for nonnegative random variables where both the independence and identical distribution conditions are weakened. In 2018, Korchevsky \cite{korchevsky2018rate} obtained a Baum-Katz type rate for the Chen-Sung Strong Law of Large Numbers \cite{CHEN201680} under stronger assumptions. This result generalised the work of Kuczmaszewska \cite{kuczmaszewska2016convergence}, in 2016, who was able to obtain rates for a Strong Law of Large Numbers result of Korchevsky in \cite{korchevsky2015generalization}, under stronger assumptions. No Baum-Katz type rates have been found for the full results in \cite{korchevsky2015generalization} and \cite{CHEN201680}.

Lastly, Baum-Katz type results can be used to obtain results concerning large deviation probabilities. For example, if $\seq{X_n}$ are iid random variables with, $\EE(X_1) = 0$ and $\mathrm{Var}(X_1) < \infty$ then condition $(iii)$ of Theorem \ref{thrm:baum-katzfull} with $r = 0$ implies,
     \begin{equation*}
        P^*_{n,\varepsilon} = o\left(\frac{1}{n}\right).
    \end{equation*}
The above is a stronger result than what one gets in (\ref{eqn:iidHR}) through the H{\'a}jek and R{\'e}nyi inequality (although one must assume more, namely that the random variables are identically distributed). Unlike the bound in  Theorem \ref{thrm:bound}, this result is ineffective in the sense that it does not explicitly tell you the constant $C$ such that $ P^*_{n,\varepsilon} \le \frac{C}{n}$, in addition, one cannot determine, a priori, that such a constant is independent of the distribution of the random variables. Furthermore, the bound in Theorem \ref{thrm:bound} only requires the assumption that the random variables are pairwise independent. 
\subsection{Proof mining}
\label{subsec:proof mining}
Applied proof theory (or proof mining) is a research area which aims to use tools and ideas from logic to extract quantitative data from proofs that appear nonconstructive. Although the program has its origins in Kreisel's ``unwinding'' program of the 1950s \cite{kreisel:51:proofinterpretation:part1,kreisel:52:proofinterpretation:part2}, its emergence as a fully substantiated subfield of applied logic was due to the work of Kohlenbach and his collaborators. This program has enjoyed a lot of success in analysis (see \cite{kohlenbach:08:book} for a comprehensive overview of the program and the recent survey papers \cite{kohlenbach:19:nonlinear:icm,kohlenbach2017recent} for applications) but has also expanded into various other areas of mathematics, where we in particular mention Tauberian theory \cite{powell2020note,powell2023finitization}, differential algebra \cite{simmons2019proof} and probability theory \cite{arthan2021borel, Avigad-Dean-Rute:Dominated:12,AVIGAD-GERHARDY-TOWSNER:10:Ergodic}, with the latter references being the only applications of the techniques of proof mining to probability theory so far. The results in this article are similarly obtained via
this logical perspective, hence, can be seen as another case study of
proof mining in probability theory, breaking ground in the extraction of
quantitative data from limit theorems.

Deriving rates of convergences from convergence results about sequences of real numbers as well as sequences taking values in abstract spaces, using these tools from logic, is a standard occupation in applied proof theory (see, for example, \cite{colao2011alternative,kohlenbach2019moduli,findling2024rates,powell2021rates}). Furthermore, one of the critical features of proof mining are the logical metatheorems (see e.g. \cite{kohlenbach2005some,gerhardy2008general,pischke2023logical,puaunescu2022proof}) that guarantee the extractability of very uniform quantitative information, such as rates of convergence, from proofs of results that can, at least in principle, be formalised in specific formal logical systems. Furthermore, these metatheorems provide exact algorithms that take formal proofs of these results and output such quantitative data.\\ 

Although, in principle, the algorithms given by the metatheorems are implementable on a computer, so the process of going from a formal proof to quantitative data is fully automatable,\footnote{This is already a feature available on the proof assistant Coq \cite{letouzey2008extraction} and the author has done work on formalising some results in the proof mining literature on the proof assistant Lean \url{https://github.com/Kejineri/Proof-mining-/}} there are drawbacks to doing this in everyday mathematical practices. For one, the process of fully formalising a mathematical proof is already completely nontrivial; furthermore, the outputted quantitative data will most probably not be readable or useful to any human, especially when a proof uses more exotic logical principles. Rather, the key feature of these algorithms that allows them to have use in everyday mathematical practice is their modularity. That is, one does not have to formalise the entire proof: If one formalises the parts of the proof that (appear to) use nonconstructive reasoning, the algorithms from the metatheorems combined with normal mathematical intuition can be used to extract (readable) quantitative content by hand. Furthermore, the metatheorems tell us how the quantitative versions of each of the lemmas required to prove the result should fit together to get a fully quantitative result. 

Jointly with Pischke \cite{NeriPischke2023}, the author has recently introduced a logical system, along with corresponding metatheorems, that for the first time provided such results for probability theory and hence shed light on the quantitative nature of probability theory as seen from this logical perspective. Looking at the Laws of Large Numbers through the lens of such a formal system has allowed for success in the extraction of the quantitative results, including the results in this article as well as upcoming work from the author \cite{Ner2023}. Furthermore, we claim that such a paradigm shift can be very beneficial in establishing new quantitative results in the context of Laws of Large Numbers and potentially even other areas in probability theory. For example, the improvement of the bound in \cite{luzia_2018} given as Theorem \ref{thrm:bound} is a testament to how analysing proofs via such logical methods allows for one not to introduce further complexity, which can be easily done when one tries to obtain computational content in a more ad hoc manner. This is because mathematical proofs inherently contain computational content,\footnote{This can be seen informally as a consequence of what is known as the Curry-Howard correspondence.} which can become obscured when one presents a proof in non-formal ``normal" mathematical language. However, this computational content becomes a lot more apparent when studying a proof more formally.\\ 

The last feature of proof mining metatheorems we would like to note is the explanatory power of observed phenomena in everyday mathematics. As briefly noted at the end of the first subsection, the bounds obtained in this paper are very uniform. Such uniformities seem to occur a lot in quantitative probability theory as
developed using this logical methodology, as was explicitly noted in the seminal paper of proof mining in probability \cite{Avigad-Dean-Rute:Dominated:12}. In \cite{NeriPischke2023}, the first logical explanation of this phenomenon by a novel extension of strong majorizability due to Bezem \cite{Bez1985} is given. In that way, also the uniformities of the present
results can be recognized as an a priori guaranteed feature of our
approach to quantitative probability theory via these logical methods.\\ 

As common in the context of applied proof theory, while
this logical background was crucial for obtaining the results of this
paper, we do not assume any familiarity with the concepts or techniques
of proof mining and, even further, none of the results or proofs
presented here, make any explicit use of such methods.
\section{A General Theorem}
\label{sec:gen}
In this section, we shall state and prove the general quantitative theorem we alluded to in Section \ref{sec:intro}. This theorem will be a quantitative version of (a generalisation of) a critical step in proving \cite[Theorem 1]{csorgHo1983strong}, which is a result that has been modified many times to obtain various ``Strong Law of Large Numbers''-type results as discussed in the introduction.\\

For this, we now first introduce the following definitions that are mostly as presented in \cite{csorgHo1983strong}: Let $\seq{X_n}$ be a sequence of nonnegative random variables with respective expected values $\seq{\mu_n}$. Let $S_n := \sum_{k=1}^n X_k$, $z_n := \sum_{i = 1}^n\mu_i$ and suppose there exists $W>0$ such that
\[
\frac{z_n}{n} \le W
\]
for all $n \in \NN$. Further, we make use of the following definitions:

\begin{itemize}
    \item For each $\delta > 0$, let $L_\delta := \floor{\frac{W}{\delta}}$.\\
    \item For each $\delta > 0$, $\alpha > 1$ and natural numbers $m$ and $0 \le s \le L_\delta$, let 
    \[
    C_{\alpha,s,\delta,m} :=\left\{\alpha^m \le n < \alpha^{m+1} \mid \frac{z_n}{n} \in [s\delta,(s+1)\delta)\right\}.
    \] 
    \item  Let $k_s^-(m):= \min C_{\alpha,s,\delta,m}$ and $k_s^+(m):= \max C_{\alpha,s,\delta,m}$ if $C_{\alpha,s,\delta,m}$ is nonempty.\\
    \item  Let  $k_s^-(m) = k_s^+(m) := \floor{\alpha^m}$ if $C_{\alpha,s,\delta,m}$ is empty.
\end{itemize}
One should note that $k_s^+(m)$ and $k_s^-(m)$ depend on $\delta, \alpha$ but (following the convention of \cite{csorgHo1983strong}) we hid this dependence to make the notation less cumbersome. We shall also adopt the convention (used in \cite{csorgHo1983strong}) that $k_s^{\pm}(m)$ being used in a relationship (an equation, an inequality, a limit, etc.) is short-hand for that relationship holding for both $k_s^+(m)$ and $k_s^-(m)$.\\

Our general theorem is now the following:
\begin{theorem}
\label{thrm:gen}
For all $\varepsilon, \delta > 0$, $\alpha > 1$ and $0 \le s \le L_\delta$, if
\begin{equation}
    \label{eqn:sumII}
    \sum_{n=1}^\infty \PP\left(\left|\frac{S_{k_s^\pm(n)} }{k_s^\pm(n)}- \frac{z_{k_s^\pm(n)}}{k_s^\pm(n)}\right| > \varepsilon\right) < \infty,
\end{equation}
then \footnote{Recall that the use of the $\pm$ notation means we are actually assuming the convergence of two sums in the premise.}
 \begin{equation*}
    \frac{S_n}{n} -\frac{z_n}{n} \to 0
\end{equation*}
almost surely.   
\end{theorem}
\begin{proof}
The Borel-Cantelli Lemma and (\ref{eqn:sumII}) implies that for all $\varepsilon, \delta > 0$, $\alpha > 1$ and all $0 \le s \le L_\delta$:
\begin{equation}
\label{eqn:conv0}
    \frac{1}{k_s^\pm(n)}S_{k_s^\pm(n)} - \frac{1}{k_s^\pm(n)}z_{k_s^\pm(n)} \to 0
\end{equation}
almost surely. For all $m \in \NN$, we can take a natural number $0 \le s \le L_{\delta}$ such that 
\begin{equation}
    \label{eqn:mdelta}
    \frac{1}{m}z_m \in [s\delta,(s+1)\delta)
\end{equation}
since $z_n/n \le W$ and $L_\delta = \floor{\frac{W}{\delta}}$. Thus, if we take $p \in \NN$ such that $\alpha^p \le m < \alpha^{p+1}$, then  $m \in C_{\alpha,s,\delta,p}$ by definition, so $C_{\alpha,s,\delta,p}$ is non-empty. Therefore, $k_s^-(p) \le m \le k_s^+(p)$ and, since  $k_s^\pm(p) \in C_{\alpha,s,\delta,p}$, we have 
$$\frac{1}{k_s^\pm(p)}z_{k_s^\pm(p)} \in [s\delta,(s+1)\delta)$$
which implies, by (\ref{eqn:mdelta}), that
\begin{equation}
\label{eqn:ledelta}
        \left|\frac{1}{m}z_m - \frac{1}{k_s^\pm(p)}z_{k_s^\pm(p)}\right| \le \delta.
    \end{equation}
Now we have the following chain of inequalities, 
\begin{equation}
\label{eqn:complicatedmaths}
\begin{aligned}
       & -\delta -\left(1-\frac{1}{\alpha}\right)W +\frac{1}{\alpha}\frac{1}{k_s^-(p)}\left(S_{k_s^-(p)} - z_{k_s^-(p)}\right) \\
       &\qquad\qquad\le -\delta - \left(1-\frac{1}{\alpha}\right)\frac{1}{k_s^-(p)}z_{k_s^-(p)} + \frac{1}{\alpha}\frac{1}{k_s^-(p)}\left(S_{k_s^-(p)} - z_{k_s^-(p)}\right)\\
       &\qquad\qquad  \le \frac{1}{m}S_{k_s^-(p)} - \frac{1}{m}z_m \\
       &\qquad\qquad  \le \frac{1}{m}(S_m - z_m)\\
       &\qquad\qquad  \le \frac{1}{m}S_{k_s^+(p)} - \frac{1}{k_s^+(p)}z_{k_s^+(p)} + \delta\\
       &\qquad\qquad  \le \frac{\alpha}{k_s^+(p)}\left(S_{k_s^+(p)} - z_{k_s^+(p)}\right) + (\alpha - 1)W + \delta.
\end{aligned}
\end{equation}
Here, the first inequality follows since
\[
\frac{1}{k_s^-(p)}z_{k_s^-(p)} < W.
\]
The second inequality follows from expanding brackets, using  (\ref{eqn:ledelta}) and the fact that $m \le \alpha k_s^-(p)$ (since $m \in C_{\alpha,s,\delta,p}$, so by definition, $m < \alpha^{p+1}$ and $k_s^-(p) \in C_{\alpha,s,\delta,p}$, and so $\alpha^p \le k_s^-(p) $). The third inequality follows from the fact that $\seq{S_n}$ is monotone (since $\seq{X_n}$ is nonnegative) and  $k_s^-(p) \le m$. The remaining inequalities are justified using similar reasoning to the above (see also \cite{csorgHo1983strong}).

Thus, by (\ref{eqn:conv0}) and the fact that $p \to \infty$ as $m \to \infty$, we have
\begin{equation*}
    -\delta -\left(1-\frac{1}{\alpha}\right)W \le \liminf_{n \to \infty}\frac{1}{m}(S_m - z_m)\le \limsup_{n \to \infty}\frac{1}{m}(S_m - z_m) \le (\alpha - 1)W + \delta
\end{equation*}
almost surely. So, taking $\delta \to 0$ and $\alpha \to 1$ gives our result.
\end{proof}

\begin{remark}
$\seq{X_n}$ (not assumed to be nonnegative) is said to converge completely to 0 if
\begin{equation*}
\sum_{n=1}^\infty \PP(|X_n| > \varepsilon) < \infty
\end{equation*}
for all $\varepsilon > 0$. Hsu and Robbins first introduced this notion of convergence in \cite{hsu1947complete}, where they demonstrated that if $\seq{X_n}$ were iid random variables with finite variance (again, not assumed to be nonnegative), then
\begin{equation*}
    \frac{S_n}{n} - \EE(X_1)
\end{equation*} 
converges to $0$ completely. Furthermore, complete convergence implies almost sure convergence by the Borel-Cantelli Lemma, so Theorem \ref{thrm:gen} says that if specifically chosen sub-sequences of
\begin{equation*}
    \frac{S_n}{n} -\frac{z_n}{n}
\end{equation*} 
converge completely to 0, then
\begin{equation*}
    \frac{S_n}{n} -\frac{z_n}{n}
\end{equation*} 
converges to $0$ almost surely.
\end{remark}
\begin{remark}
To prove \cite[Theorem 1]{csorgHo1983strong}, it is shown that 
\begin{equation}
\label{eqn:CTsum}
\sum_{n=1}^\infty \mathbb{E}\left(\left(\frac{S_{k_s^\pm(n)}}{k_s^\pm(n)} - \frac{z_{k_s^\pm(n)}}{k_s^\pm(n)}\right)^2\right) < \infty.
\end{equation}
(\ref{eqn:sumII}) in Theorem \ref{thrm:gen} follows from this by Chebyshev's inequality, so the result in \cite{csorgHo1983strong} follows by our theorem. Therefore, Theorem \ref{thrm:gen} generalises the key step in proving \cite[Theorem 1]{csorgHo1983strong}.
\end{remark}

We now give a quantitative version of Theorem \ref{thrm:gen}:
\begin{theorem}
\label{thrm:genquant}
Suppose for each $\varepsilon, \delta > 0$, $\alpha > 1$ and $0 \le s \le L_\delta$:
\begin{equation}
\sum_{n=1}^\infty \PP\left(\left|\frac{S_{k_s^\pm(n)}}{k_s^\pm(n)} - \frac{z_{k_s^\pm(n)}}{k_s^\pm(n)}\right| > \varepsilon\right) < \infty.
\end{equation}
 Furthermore, suppose that the partial sums of both sums converge to their respective limits with a rate of convergence  $\Lambda_{\varepsilon, \delta,\alpha}: \RR \to \RR$, independent of $s$.\footnote{We can always obtain a rate independent of $s$ by taking the maximum value of all such rates that depend on $s$, as $s$ can only take the value of finitely many natural numbers. Furthermore, if both sums (the plus one and the minus one) have different rates, we can obtain one that works for both by taking the maximum of the two rates.} 
 
 More explicitly, for each $\varepsilon, \delta > 0$, $\alpha > 1$, $0 \le s \le L_\delta$, $\lambda > 0$ and $p \ge \Lambda_{\varepsilon, \delta,\alpha}$, we have,
 \begin{equation*}
\sum_{n=p+1}^\infty \PP\left(\left|\frac{S_{k_s^-(n)}}{k_s^-(n)} - \frac{z_{k_s^-(n)}}{k_s^-(n)}\right| > \varepsilon\right)\le \lambda\\ \mbox{ and } \sum_{n=p+1}^\infty \PP\left(\left|\frac{S_{k_s^+(n)}}{k_s^+(n)} - \frac{z_{k_s^+(n)}}{k_s^+(n)}\right| > \varepsilon\right)\le \lambda.
 \end{equation*}
Then, for all $\varepsilon > 0$,
$$ \PP\left(\sup_{m \ge n} \left|\frac{S_m}{m} -\frac{z_m}{m}\right| > \varepsilon\right) \to 0$$  
with a rate of convergence given by 
$$\Phi_{\varepsilon,\Lambda}(\lambda) := \alpha^{ \Pi_\varepsilon(\lambda)},$$
where
\[
\Pi_\varepsilon(\lambda):= \Lambda_{\frac{\varepsilon}{3\alpha},\frac{\varepsilon}{3},\alpha} \left(\frac{\lambda}{2}\right) + 1\text{ and }\alpha := 1 + \frac{\varepsilon}{3W}.
\]
\end{theorem}
\begin{proof}
  First we observe that, for all $\delta,\lambda,\varepsilon>0$, $\alpha > 1$, natural numbers $0 \le s \le L_\delta$ and $p \ge \Lambda_{\varepsilon, \delta,\alpha}(\lambda) + 1 $:
        \begin{equation*}
        \begin{aligned}
         \PP\left(\sup_{q \ge p}\left|\frac{S_{k_s^\pm(q)}}{k_s^\pm(q)} - \frac{z_{k_s^\pm(q)}}{k_s^\pm(q)}\right|>\varepsilon\right)
        &=\PP\left(\bigcup_{q = p}^{\infty} \left(\left|\frac{S_{k_s^\pm(q)}}{k_s^\pm(q)} - \frac{z_{k_s^\pm(q)}}{k_s^\pm(q)}\right| > \varepsilon\right)\right) \\
      &\le \sum_{q=p}^{\infty} \PP\left(\left|\frac{S_{k_s^\pm(q)}}{k_s^\pm(q)} - \frac{z_{k_s^\pm(q)}}{k_s^\pm(q)}\right| > \varepsilon\right)\le \lambda.
    \end{aligned}
    \end{equation*}
    Here, the last inequality follows from the fact that $p - 1 \ge \Lambda_{\varepsilon, \delta,\alpha}(\lambda) $ (and that $\Lambda$ is a rate of convergence).
    
    Now, fix  $\varepsilon, \lambda> 0$ and 
    $$ n \ge \Phi_{\varepsilon,\Lambda}(\lambda) = \alpha^{\Lambda_{\frac{\varepsilon}{3\alpha},\frac{\varepsilon}{3},\alpha} (\frac{\lambda}{2}) + 1}.$$
    We must show,
    $$\PP\left(\sup_{m \ge n}\left|\frac{S_m}{m} -  \frac{z_m}{m}\right|>\varepsilon\right) \le \lambda.$$
     Set $\delta = \frac{\varepsilon}{3}$ and observe that having $\alpha = 1 + \frac{\varepsilon}{3W}$ ensures that 
    \begin{equation}
        \label{eqn:choicealpha}
            -\frac{\varepsilon}{3} \le - (1-\frac{1}{\alpha})W\text{ and }(\alpha - 1)W = \frac{\varepsilon}{3}.
    \end{equation}
Take $p \in \NN$ such that $\alpha^p \le n < \alpha^{p+1}$. Then we have
$$ \alpha^{\Lambda_{\frac{\varepsilon}{3\alpha},\frac{\varepsilon}{3},\alpha} (\frac{\lambda}{2}) + 1} \le n < \alpha^{p+1}$$
which implies
$$p \ge \Lambda_{\frac{\varepsilon}{3\alpha},\frac{\varepsilon}{3},\alpha} \left(\frac{\lambda}{2}\right) + 1.$$
Thus, by the very first step of the proof, we have, for each $0\le r\le L_\delta$,
\begin{equation}
\label{eqn:pstep}
    \PP\left(\sup_{q \ge p}\left|\frac{1}{k_r^\pm(q)}S_{k_r^\pm(q)} - \frac{1}{k_r^\pm(q)}z_{k_r^\pm(q)}\right| > \frac{\varepsilon}{3\alpha}\right)\le \frac{\lambda}{2}.
\end{equation}
Thus, it suffices to show that there exists $0\le r\le L_\delta$ such that  
$$\sup_{m \ge n}\left|\frac{S_m}{m} -  \frac{z_m}{m}\right|>\varepsilon$$
implies that
$$\sup_{q \ge p}\left|\frac{1}{k_r^-(q)}S_{k_r^-(q)} - \frac{1}{k_r^-(q)}z_{k_r^-(q)}\right| > \frac{\varepsilon}{3\alpha}$$
or that
$$\sup_{q \ge p}\left|\frac{1}{k_r^+(q)}S_{k_r^+(q)} - \frac{1}{k_r^+(q)}z_{k_r^+(q)}\right| > \frac{\varepsilon}{3\alpha},$$
as then we would have, for such an $r$,

\begin{equation*}
    \begin{aligned}
        \PP\left(\sup_{m \ge n}\left|\frac{S_m}{m} -  \frac{z_m}{m}\right|>\varepsilon\right) 
        &\le \PP\left(\sup_{q \ge p}\left|\frac{1}{k_r^-(q)}S_{k_r^-(q)} - \frac{1}{k_r^-(q)}z_{k_r^-(q)}\right| > \frac{\varepsilon}{3\alpha}\right)\\ 
        &+ \PP\left(\sup_{q \ge p}\left|\frac{1}{k_r^+(q)}S_{k_r^+(q)} - \frac{1}{k_r^+(q)}z_{k_r^+(q)}\right| > \frac{\varepsilon}{3\alpha}\right) \le \lambda,
    \end{aligned}
\end{equation*}
which is what we are required to show (with the final inequality following from (\ref{eqn:pstep})).

Suppose, for contradiction, that the above was not the case. Then, for all $0\le r\le L_\delta$, we have
$$\sup_{m \ge n}\left|\frac{1}{m}S_m -  \frac{1}{m}z_m\right|>\varepsilon$$
and
\begin{equation}
\label{eqn:contra}
     \left|\frac{1}{k_r^\pm(q)}S_{k_r^\pm(q)} - \frac{1}{k_r^\pm(q)}z_{k_r^\pm(q)}\right| \le \frac{\varepsilon}{3\alpha}
\end{equation}
for all $q \ge p$. Take $m \ge n$ such that 
\begin{equation}
    \label{eqn:mstep}
    \left|\frac{1}{m}S_m -  \frac{1}{m}z_m\right|>\varepsilon.
\end{equation}
We now use arguments similar to the proof of Theorem \ref{thrm:gen}. We can find $0\le r\le L_\delta$ such that 
$$\frac{1}{m}z_m \in [r\delta,(r+1)\delta],$$
 so, taking $q \in \NN$ such that $\alpha^q \le m < \alpha^{q+1}$, ensures that $m \in C_{\alpha,r,\delta,q}$. Furthermore, as $m \ge n$, we have $q \ge p$.

Now, since  $k_r^\pm(q) \in C_{\alpha,r,\delta,q}$, we have 
$$\frac{1}{k_r^\pm(q)}z_{k_r^\pm(q)} \in [r\delta,(r+1)\delta)$$
which implies
$$\left|\frac{1}{m}z_m - \frac{1}{k_r^\pm(q)}z_{k_r^\pm(q)}\right| \le \delta.$$
 Now, following the exact same reasoning as (\ref{eqn:complicatedmaths}), we have
\begin{equation*}
\begin{aligned}
       &-\delta -\left(1-\frac{1}{\alpha}\right)W +\frac{1}{\alpha}\frac{1}{k_r^-(q)}\left(S_{k_r^-(q)} - z_{k_r^-(q)}\right) \\
       &\qquad\qquad\le \frac{1}{m}(S_m - z_m) \\
       &\qquad\qquad\le \frac{\alpha}{k_r^+(q)}\left(S_{k_r^+(q)} - z_{k_r^+(q)}\right) + (\alpha - 1)W + \delta.
\end{aligned}
\end{equation*}
So, (\ref{eqn:choicealpha}) implies that (recalling that $\delta = \varepsilon/3)$,
\begin{equation}
\label{eqn:simpcomp}
\begin{aligned}
     -\frac{2\varepsilon}{3} +\frac{1}{\alpha}\frac{1}{k_r^-(q)}\left(S_{k_r^-(q)} - z_{k_r^-(q)}\right)&\le\frac{1}{m}(S_m - z_m)\\
     &\le\frac{\alpha}{k_r^+(q)}\left(S_{k_r^+(q)} - z_{k_r^+(q)}\right) + \frac{2\varepsilon}{3}.
\end{aligned}
\end{equation}
Now, the above and (\ref{eqn:contra}) (and the fact that $\alpha>1$) implies $\left|1/m(S_m - z_m)\right|\le\varepsilon$, which contradicts (\ref{eqn:mstep}).
\end{proof}
\begin{remark}
Applying Theorem \ref{thrm:genquant} to obtain explicit rates of convergences for results related to the Strong Laws of Large Numbers requires that we must find an explicit rate of convergence for the (\ref{eqn:sumII}). It is a well known result in computability theory that the Monotone Convergence Theorem is computationally ineffective, in the sense that a computable bound on a series does not necessarily yield a computable rate of convergence (a monotone bounded sequence can converge arbitrarily slowly, see \cite{specker:49:sequence}). Since this step still features in the proof of Theorem \ref{thrm:gen}, we require a rate of convergence for (\ref{eqn:sumII}) in Theorem \ref{thrm:genquant} to provide a construction for a rate of convergence for the conclusion. Furthermore, when the strategy captured by Theorem \ref{thrm:gen} is applied to obtain results related to the Strong Law of Large Numbers, one typically just bounds (\ref{eqn:sumII}), thus more work will be required than what is offered by the proof of these results as we will have to explicitly calculate rates of convergences. This will be the case when we apply Theorem \ref{thrm:genquant} to prove Theorem \ref{thrm:quantchen} in Section \ref{sec:chen-sung}.  
\end{remark}
\section{Application I: Pairwise Independent with Finite Variance}
\label{sec:finvar}
In this section, we shall prove Theorem \ref{thrm:bound}. First, we calculate a rate under the assumption that the random variables are nonnegative.\\

Fix a sequence of nonnegative, pairwise independent random variables $\seq{Y_n}$ with, $\EE(Y_n) \le \mu \neq 0$, $\mathrm{Var}(Y_n)\le  \sigma_Y^2$ for all $n \in \NN$ and some $\mu, \sigma_Y > 0$. Set $S^Y_n:= \sum_{i=1}^nY_i$ and $z^Y_n := \sum_{i = 1}^n\EE(Y_n)$.
\begin{lemma}
     For all $\varepsilon, \delta> 0$, $\alpha > 1$ and $0\le s \le L_\delta$,
     $$ R_{\varepsilon,\alpha} (\lambda) = \log_\alpha\left(\frac{2\sigma_Y^2}{\lambda\varepsilon^2(\alpha-1)}\right) - 1$$
     is a rate of convergence for the partial sums of
\begin{equation}
    \sum_{n=1}^\infty \PP\left(\left|\frac{S^Y_{k_s^\pm(n)}}{k_s^\pm(n)} - \frac{z^Y_{k_s^\pm(n)}}{k_s^\pm(n)}\right| > \varepsilon\right) 
\end{equation}
to their respective limits.
\end{lemma}
\begin{proof}
   Fix $\lambda,\varepsilon, \delta> 0$ and $\alpha > 1$ as well as $0\le s \le L_\delta$, $Q \ge R_{\varepsilon,\alpha,} (\lambda)$. Then:
    \begin{equation*}
    \begin{aligned}
     \sum_{n=Q+1}^{\infty} \PP\left(\left|\frac{1}{k_s^\pm(n)}S^Y_{k^\pm(n)} - \frac{1}{k_s^\pm(n)}z^Y_{k_s^\pm(n)}\right| > \varepsilon\right) &\le \frac{1}{\varepsilon^2}\sum_{n =Q+1}^{\infty} \frac{\mathrm{Var}(S^Y_{k^\pm(n)})}{k_s^\pm(n)^2} \\
     &\le \frac{ \sigma_Y^2 }{\varepsilon^2}\sum_{n =Q+1}^{\infty}\frac{1}{k_s^\pm(n)}\\
     &\le \frac{ 2\sigma_Y^2 }{\varepsilon^2}\sum_{n =Q+1}^{\infty}\alpha^{-n}\\
     &\le \frac{ 2\sigma_Y^2\alpha^{-(Q+1)} }{\varepsilon^2(\alpha -1)} \le \lambda.
    \end{aligned}
    \end{equation*}
We get the first inequality from Chebyshev's inequality, the second inequality by pairwise independence, the third inequality by using $k^\pm(n) \ge \floor{\alpha^n} > \alpha^n/2$, the fourth inequality by using the sum of an infinite geometric sequence and the last inequality from the assumption that $Q \ge R_{\varepsilon,\alpha} (\lambda)$.
\end{proof}
We can now apply Theorem \ref{thrm:genquant} with the rate we obtained above, observing that $R$ is independent of $s$ (and $\delta$), to easily obtain the following:
\begin{lemma}
    For all $\varepsilon, \lambda> 0$ and all
    $$ n \ge \Delta_{\varepsilon,\mu,\sigma_Y}(\lambda) := \Phi_{\varepsilon,R} (\lambda) := \frac{36\alpha^2\sigma_Y^2}{\lambda\varepsilon^2(\alpha-1)},$$
    it holds that
        \begin{equation*}
         \PP\left(\sup_{m \ge n}\left|\frac{S^Y_m}{m} - \frac{z_{k_s^\pm(n)}}{k_s^\pm(n)}\right|>\varepsilon\right) \le \lambda.
    \end{equation*}
    Here, $\alpha := 1 + \frac{\varepsilon}{3\mu}$, $R$ is defined as in the previous lemma and $\Phi$ is defined as in Theorem \ref{thrm:genquant}.
\end{lemma}
\begin{proof}
    This follows immediately from Theorem \ref{thrm:genquant}. Note we may take $W$ to be $\mu$.
\end{proof}
We can now obtain a rate where the random variables are not assumed to be nonnegative.
\begin{proposition}
\label{thrm:betterrate}
Let $\seq{X_n}$ be a sequence of pairwise independent random variables with $\EE(X_n) = 0$, $\mathrm{Var}(X_n) \le \sigma^2$ and $\EE(|X_n|) \le \tau$ for all $n \in \NN$ and some $\tau, \sigma > 0$. Furthermore, let $S_n := \sum_{i=1}^nX_i$. Then for all $\varepsilon, \lambda> 0$ and all $n \ge \Delta_{\frac{\varepsilon}{2},\frac{\tau}{2},\sigma}(\frac{\lambda}{2}) $:
        \begin{equation*}
         P^*_{n,\varepsilon} = \PP\left(\sup_{m \ge n}\left|\frac{S_m}{m}\right|>\varepsilon\right) \le \lambda.
    \end{equation*}
    Here,$$\Delta_{\frac{\varepsilon}{2},\frac{\tau}{2},\sigma}\left(\frac{\lambda}{2}\right) := \frac{288\alpha^2\sigma^2}{\lambda\varepsilon^2(\alpha-1)}$$
    as before with $\alpha := 1 + \frac{\varepsilon}{3\tau}$. Thus, $P^*_{n,\varepsilon}$ converges to 0 with a rate of convergence given above.
\end{proposition}
\begin{proof}
    We have, for all $n \in \NN$, $\sigma^2 \ge  \mathrm{Var}(X_n) = \EE(X_n^2) \ge \mathrm{Var}(X_n^+) + \mathrm{Var}(X_n^-)$ which implies $\sigma^2 \ge \mathrm{Var}(X_n^\pm)$. Furthermore, we have $\EE(X_n^\pm) \le \frac{\tau}{2}$ since $\EE(X_n) = 0 =\EE(X_n^+)-\EE(X_n^-)$. Thus, if we take $\seq{Y_n} = \seq{X_n^\pm}$, we can set $\sigma_Y:= \sigma$ and $\mu:= \tau/2$. Furthermore we can set
    \begin{equation*}
        z_n := \sum_{i=1}^n \EE(X_n^+) = \sum_{i=1}^n \EE(X_n^-). 
    \end{equation*}
    Thus, from the previous lemma:
        \begin{equation*}
         \PP\left(\sup_{m \ge n}\left|\frac{1}{m}S^\pm_m - \frac{1}{m}z_m\right|>\frac{\varepsilon}{2}\right) \le \frac{\lambda}{2}
    \end{equation*}
    for all $\varepsilon, \lambda> 0$ and $n \ge \Delta_{\frac{\varepsilon}{2},\mu,\sigma_Y}(\frac{\lambda}{2})$. Thus, if $ n \ge \Delta_{\frac{\varepsilon}{2},\frac{\tau}{2},\sigma}(\frac{\lambda}{2}) = \Delta_{\frac{\varepsilon}{2},\mu,\sigma_Y}(\frac{\lambda}{2})$, then
    \begin{equation*}
    \begin{aligned}
    \PP\left(\sup_{m \ge n} \left|\frac{1}{m}S_{m}\right| > \varepsilon\right) &= \PP\left(\sup_{m \ge n} \left|\left(\frac{1}{m}S_{m}^+ - \frac{1}{m}z_m\right) - \left(\frac{1}{m}S_{m}^- - \frac{1}{m}z_m\right)\right| > \varepsilon\right)  \\
    &\le \PP\left(\sup_{m \ge n} \left|\frac{1}{m}S_{m}^+ - \frac{1}{m}z_m\right| > \frac{\varepsilon}{2}\right) + \PP\left(\sup_{m \ge n} \left|\frac{1}{m}S_{m}^- - \frac{1}{m}z_m\right| > \frac{\varepsilon}{2}\right) \\
    &\le \lambda.
    \end{aligned}
\end{equation*}
Here $S_n^\pm := \sum_{i=1}^n X_n^\pm$.
\end{proof}
This, in particular, allows us rather immediately to deduce Theorem \ref{thrm:bound}:
\begin{proof}[Proof of Theorem \ref{thrm:bound}]
Using the above proposition, we have
\begin{equation*}
   P^*_{n,\varepsilon} \le  \frac{288\alpha^2\sigma^2}{n\varepsilon^2(\alpha-1)}
\end{equation*}
for all $n \in \NN$. So Theorem \ref{thrm:bound} follows by noting that if $\varepsilon \le \tau$, we will have $\alpha \le 4/3$.
\end{proof}
\begin{remark}
    In the case $\varepsilon > \tau$, observe that $\alpha < 4\varepsilon/3\tau$ and so we can deduce
    \begin{equation*}
      P^*_{n,\varepsilon} \le  \frac{1536\sigma^2}{n\varepsilon\tau}.
    \end{equation*}
\end{remark}

We shall now discuss the optimality of the bound we obtained.

\begin{example}
     For $\delta > 0$, let $\seq{X_n}$ be a sequence of integer-valued iid random variables such that
     \begin{equation*}
         \PP(X_1 = n) = \frac{c}{n^{3+\delta}} \\\ \mbox{for }\\ c = \left(\sum_{n = 1}^\infty \frac{1}{n^{3+\delta}}\right)^{-1}
     \end{equation*}
     for all $n \in \ZZ^+$ (and probability $0$ for all other integers).
      Then $\mathrm{Var}(X_1) < \infty$ and for all $1 \ge \varepsilon > 0$ and any $n \in \NN$:
     \begin{equation*}
    \PP\left(\sup_{m \ge n} \left|\frac{S_{m}}{m} - \mu\right| > \varepsilon\right) \ge \frac{\omega}{n^{1+\delta}}
     \end{equation*}
      where $\mu$ is the mean of $X_1$, given by 
     \begin{equation*}
         \mu=\sum_{n = 1}^\infty \frac{c}{n^{2+\delta}},
     \end{equation*}
     and
     \[
     \omega= \frac{c}{2\times3^{2+\delta}(2+\delta)}.
     \]
\end{example}
\begin{proof}
These random variables clearly have finite variance. For any $1 \ge \varepsilon > 0$, we have
\begin{equation*}
    \PP\left(\sup_{m \ge n} \left|\frac{S_{m}}{m} - \mu\right| > \varepsilon\right) \ge \PP\left(\sup_{m \ge n} \left|\frac{S_{m}}{m} - \mu\right| \ge 1 \right). 
\end{equation*}
Observe that
$$\mu = \frac{\sum_{n = 1}^\infty \frac{1}{n^{2+\delta}}}{\sum_{n = 1}^\infty \frac{1}{n^{3+\delta}}} < \frac{\zeta(2)}{\zeta(4)} = \frac{15}{\pi^2} < 2. $$
Thus, we get
\begin{equation*}
\begin{aligned}
    \PP\left(\sup_{m \ge n} \left|\frac{S_{m}}{m} - \mu\right| \ge 1 \right) &\ge \PP\left(\sup_{m \ge n} \frac{S_{m}}{m} \ge 3\right) \\
    &\ge \PP\left(\frac{S_{n}}{n} \ge 3\right)\\
    &\ge\PP(X_1 \ge 3n \cup\ldots\cup X_n \ge 3n)\\
    &= 1- \PP(X_1 < 3n \cap \ldots \cap X_n < 3n)\\
    &=1-(\PP(X_1 < 3n))^n = 1-(1-\PP(X_1 \ge 3n))^n. 
\end{aligned}
\end{equation*}
We now have
\begin{equation*}
    \PP(X_1 \ge 3n) = c\sum_{k = 3n}^\infty \frac{1}{k^{3+\delta}}\ge c\int_{3n}^\infty \frac{1}{x^{3+\delta}} \,dx = \frac{c}{(3n)^{2+\delta}(2+\delta)} = \frac{w}{n^{2+\delta}},
\end{equation*}
where $w = \frac{c}{3^{2+\delta}(2+\delta)}$. This implies,
\begin{gather*}
 \PP\left(\sup_{m \ge n} \left|\frac{S_{m}}{m} - \mu\right| > \varepsilon\right) \ge 1-\left(1-\frac{w}{n^{2+\delta}}\right)^n\ge\\
 1-\frac{1}{1+\frac{w}{n^{1+\delta}}}\ge\frac{w}{2n^{1+\delta}},
\end{gather*}
where we used the inequality $(1+x)^n \le \frac{1}{1+nx}$ and $w < 1$. This yields the result.
\end{proof}

Therefore, for every $\delta > 0$, by translation, we can obtain a sequence of iid random variables with expected values equal to $0$ and finite variance, such that
     \begin{equation*}
    \PP\left(\sup_{m \ge n} \left|\frac{1}{m}S_{m}\right| > \varepsilon\right) \ge \frac{\omega}{n^{1+\delta}}.
     \end{equation*}
This example demonstrates that $O\left(\frac{1}{n}\right)$ is an optimal general power of $n$ bound for $P^*_{n,\varepsilon}$ in the case of finite variance. It however does not rule out the possibility that $P^*_{n,\varepsilon} = O\left(\frac{1}{n\log(n)}\right)$, for example.
\section{Application II: The Chen-Sung Strong Law of Large Numbers }
\label{sec:chen-sung}
Throughout this section, let $\seq{X_n}$ be a sequence of random variables with finite $p$-th moment (for some fixed $p\ge 1$) and respective means $\seq{\mu_n}$. Let $S_n := \sum_{k=1}^n X_k$ and $z_n := \sum_{i = 1}^n\mu_i$.\\

To use Theorem \ref{thrm:genquant} to obtain a quantitative version of Theorem \ref{thrm:chen}, we must find a rate of convergence for (\ref{eqn:sumII}). To do this, we need some lemmas. The first is a technical lemma that resembles the  H{\'a}jek and R{\'e}nyi inequality.
\begin{lemma}
\label{lem:sumbound}
 Suppose $\seq{\gamma_n}$ is a sequence of nonnegative real numbers satisfying
$$\EE(|S_n - z_n|^p) \le \sum_{k=1}^n \gamma_k.$$
For all $\varepsilon, \delta> 0$, $\alpha > 1$ and $0\le s \le L_\delta$:
      \begin{equation}
      \label{eqn:2parts}
    \begin{aligned}
        &\sum_{n=Q + 1}^\infty \PP\left(\left|\frac{S_{k_s^\pm(n)}}{k_s^\pm(n)} - \frac{z_{k_s^\pm(n)}}{k_s^\pm(n)}\right| > \varepsilon\right) \\
        &\qquad\qquad\le \frac{2^p\alpha^{2p}}{\varepsilon^p\floor{\alpha^{Q+2}}^p(\alpha^p - 1)}\sum_{m =1}^{\floor{\alpha^{Q+2}}} \gamma_m  +\frac{2^p\alpha^{2p}}{\varepsilon^p(\alpha^p - 1)}\sum_{m =\floor{\alpha^{Q+1}}+1}^\infty \frac{\gamma_m}{m^p}.
    \end{aligned} 
    \end{equation}
\end{lemma}
\begin{proof}
   Fix $M \in \NN$. By the generalised Chebyshev's inequality, we have
    \begin{equation*}
    \begin{aligned}
        \sum_{n=Q+1}^{M} \PP\left(\left|\frac{S_{k_s^\pm(n)}}{k_s^\pm(n)} - \frac{z_{k_s^\pm(n)}}{k_s^\pm(n)}\right| > \varepsilon\right) &\le \frac{1}{\varepsilon^p}\sum_{n =Q+1}^{M} \frac{\EE\left(\left|S_{k_s^\pm(n)} - z_{k_s^\pm(n)}\right|^p\right) }{k_s^\pm(n)^p} \\
    &\le\frac{1}{\varepsilon^p}\sum_{n =Q+1}^{M}\frac{1}{k_s^\pm(n)^p} \sum_{m=1}^{k_s^\pm(n)} \gamma_m. 
    \end{aligned}    
    \end{equation*}
    Now, splitting the inner sum into two parts, observing that if  $n > Q+1$ then $k_s^\pm(n) > k_s^\pm(Q+1)$, we have that the above sum is equal to
    \begin{equation}
    \label{eqn:sum2parts}
       \frac{1}{\varepsilon^p}\sum_{n =Q+1}^{M}\frac{1}{k_s^\pm(n)^p} \sum_{m=1}^{k_s^\pm(Q+1)} \gamma_m  + \frac{1}{\varepsilon^p}\sum_{n =Q+1}^{M}\frac{1}{k_s^\pm(n)^p} \sum_{m=k_s^\pm(Q+1)+1}^{k_s^\pm(n)} \gamma_m.
    \end{equation}
    Now, interchanging summations in the first term and using $k_s^\pm(n) \ge \floor{\alpha^n} > \alpha^n/2$, we can bound the first term from above by
    \begin{equation*}
    \begin{aligned}
        \frac{2^p}{\varepsilon^p}\sum_{m =1}^{k_s^\pm(Q+1)} \gamma_m \sum_{n=Q+1}^{M}\alpha^{-pn} &\le \frac{2^p\alpha^p}{\varepsilon^p\alpha^{p(Q+1)}(\alpha^p - 1)}\sum_{m =1}^{k_s^\pm(Q+1)} \gamma_m\\
        &\le \frac{2^p\alpha^{2p}}{\varepsilon^p\alpha^{p(Q+2)}(\alpha^p - 1)}\sum_{m =1}^{k_s^\pm(Q+1)} \gamma_m \\
        &\le\frac{2^p\alpha^{2p}}{\varepsilon^p\floor{\alpha^{Q+2}}^p(\alpha^p - 1)}\sum_{m =1}^{\floor{\alpha^{Q+2}}} \gamma_m. 
    \end{aligned}
    \end{equation*}
    We bound the first line by an infinite geometric series to get the second line, and we use $\alpha^{Q+2} > k_s^\pm(Q+1)$ to get from the penultimate line to the last line.
    
    We now bound the second term in (\ref{eqn:sum2parts}) from above. Again, interchanging summations and using similar manipulations as used to obtain the bound for the first term, we get that the second is bounded above by
     \begin{equation*}
     \begin{aligned}
         &\frac{1}{\varepsilon^p}\sum_{m =k_s^\pm(Q+1)+1}^{k_s^\pm(M)} \gamma_m \sum_{\{M\ge n \ge Q+1 : k_s^\pm(n) \ge m\}}\frac{1}{k_s^\pm(n)^p}\\
         &\qquad\qquad\le \frac{2^p}{\varepsilon^p}\sum_{m =k_s^\pm(Q+1)+1}^{k_s^\pm(M)} \gamma_m \sum_{\{M\ge n \ge Q : \alpha^{n+1} \ge m\}}\alpha^{-pn}.
     \end{aligned}.
    \end{equation*}
    The inner sum is bounded by an infinite geometric series with the first term $\le m^{-p}\alpha^p$. Thus, the above is again bounded above by
     \begin{equation*}
     \begin{aligned}
         \frac{2^p\alpha^{2p}}{\varepsilon^p(\alpha^p - 1)}\sum_{m =k_s^\pm(Q+1)+1}^{k_s^\pm(M)} \frac{\gamma_m}{m^p} \le \frac{2^p\alpha^{2p}}{\varepsilon^p(\alpha^p - 1)}\sum_{m =\floor{\alpha^{Q+1}}+1}^{\floor{\alpha^{M+1}}} \frac{\gamma_m}{m^p}.
     \end{aligned}
    \end{equation*}
Taking $M \to \infty$ gives the required result.
\end{proof}
To find a rate of convergence for (\ref{eqn:sumII}), we must find one for the two terms on the right-hand side of (\ref{eqn:2parts}). A rate for the second term can easily be calculated given one for  $\sum_{m = 1}^\infty \frac{\gamma_m}{m^p}$. To obtain a rate for the second term, we need a quantitative version of what is known as Kronecker's Lemma. The author has established a more general quantitative version of Kronecker's lemma in \cite{Ner2023}, which is used to establish a quantitative version of the analogue of Chung's Strong Law of Large Numbers \cite{chung:47:LLN} on Banach spaces given in \cite{woyczynski1974random}. However, we here state and prove the special case of this result that is required for our current purposes. One can turn to \cite[Theorem $A.6.2$]{gut2005probability} for a proof of the non-quantitative result. 

\begin{lemma}[Quantitative Kronecker's lemma]
    \label{lem:kronecker}
    Let $x_1, x_2,\ldots$ be a sequence of nonnegative real numbers such that $\sum_{i=1}^\infty x_i < \infty$  and let $0< a_1 \le a_2 \le \ldots$ be such that $a_n \to \infty$. Quantitatively, suppose $\sum_{i=1}^\infty x_i < S$ for some $S > 0$ and that $s_n := \sum_{i=1}^n x_i$ converges to $\sum_{i=1}^\infty x_i$ with rate of convergence $\phi$. Further, suppose that there is a function $f: \RR \to \NN$ such that $a_{f(\omega)} \ge \omega$ for all $\omega > 0$. Then 
    \begin{equation*}
        \frac{1}{a_n}\sum_{i=1}^n a_ix_i \to 0
    \end{equation*}
    as $n \to \infty$ with rate of convergence 
       \begin{equation*}
        K_{\phi,f,\seq{a_n},S}(\varepsilon) = \max\left\{\phi\left(\frac{\varepsilon}{4}\right),f\left(\frac{4a_{\phi(\frac{\varepsilon}{4})}S}{\varepsilon}\right)\right\}.
    \end{equation*}
\end{lemma}
\begin{proof}
    Take $\varepsilon > 0$ and $n \ge  K_{\phi,f,\seq{a_n},S}(\varepsilon)$ and let $M = \phi(\frac{\varepsilon}{4})$. From the definition of $\phi$, we have for all $i \ge M$ that $|s_M - s_i| \le \frac{\varepsilon}{4}$. Summation by parts gives
    \begin{equation*}
        \left|\frac{1}{a_n}\sum_{i=1}^n a_ix_i\right| = \left|s_n - \frac{1}{a_n}\sum_{i=1}^{n-1} (a_{i+1}-a_i)s_i\right| 
    \end{equation*}
which in turn is equal to
    \begin{equation*}
       \begin{aligned}
           &\left|s_n - \frac{1}{a_n}\sum_{i=1}^{M-1} (a_{i+1}-a_i)s_i - \frac{1}{a_n}\sum_{i=M}^{n-1} (a_{i+1}-a_i)s_M - \frac{1}{a_n}\sum_{i=M}^{n-1} (a_{i+1}-a_i)(s_i-s_M)\right|\\
           &\qquad \le \left|s_n -(1-\frac{a_M}{a_n})s_M\right| + \left|\frac{1}{a_n}\sum_{i=1}^{M-1} (a_{i+1}-a_i)s_i\right| + \left|\frac{1}{a_n}\sum_{i=M}^{n-1} (a_{i+1}-a_i)(s_i-s_M)\right| \\
           &\qquad\le |s_n - s_M| +\left|\frac{a_Ms_M}{a_n}\right|+ \left|\frac{1}{a_n}\sum_{i=1}^{M-1} (a_{i+1}-a_i)s_i\right| + \left|\frac{1}{a_n}\sum_{i=M}^{n-1} (a_{i+1}-a_i)(s_i-s_M)\right| \\
           &\qquad\le \frac{\varepsilon}{4}  +\left|\frac{a_MS}{a_n}\right|+ \left|\frac{S}{a_n}\sum_{i=1}^{M-1} (a_{i+1}-a_i)\right| + \left|\frac{1}{a_n}\frac{\varepsilon}{4}\sum_{i=M}^{n-1} (a_{i+1}-a_i)\right|\\
           &\qquad\le\frac{\varepsilon}{4}  +\left|\frac{a_MS}{a_n}\right|+ \left|\frac{Sa_M}{a_n}\right| + \frac{\varepsilon}{4} \le \varepsilon.
       \end{aligned}
    \end{equation*}
\end{proof}
 We can now calculate a rate of convergence for (\ref{eqn:sumII})
 \begin{lemma}\label{lem:chiConstr}
    Suppose $\seq{X_n}$ and $\seq{\gamma_n}$ are as in Theorem \ref{thrm:chen}. Suppose $\sum_{m = 1}^\infty \frac{\gamma_n}{m^p} \le \Gamma$ for some $\Gamma > 0$ and that the partial sums converge to their limit with a strictly decreasing rate of convergence $\Psi$. For all $\varepsilon, \delta> 0$, $\alpha > 1$ and $0\le s \le L_\delta$, the function $\chi_{\varepsilon,\alpha, \Psi}$ is a rate of convergence for the partial sums of
     \begin{equation*}
          \sum_{m=1}^\infty \PP\left(\left|\frac{S_{k_s^\pm(m)}}{k_s^\pm(m)} - \frac{z_{k_s^\pm(m)}}{k_s^\pm(m)}\right| > \varepsilon\right) 
     \end{equation*}
     to their respective limits, where
     \begin{equation*}
             \chi_{\varepsilon,\alpha, \Psi}(\lambda) =  \max\left\{\log_\alpha\left(2\Psi\left(\frac{\lambda\varepsilon^p(\alpha^p -1)}{2^{p+1}\alpha^{2p}}\right)\right), \log_\alpha\left(2K_{\psi,f_p,\seq{n^p}, R}\left(\frac{\lambda}{2}\right)\right)\right\}
     \end{equation*}
     with
     \[
             \psi(\lambda) = \Psi\left(\frac{\lambda\varepsilon^p(\alpha^p - 1)}{2^p\alpha^{2p}}\right),\quad f_p(\omega) = \ceil{\omega^{\frac{1}{p}}},\quad R = \frac{2^p\Gamma\alpha^{2p}}{\varepsilon^p(\alpha^p - 1)}.
     \]
 \end{lemma}
 \begin{proof}
     Let $\lambda,\varepsilon, \delta> 0$, $\alpha > 1$ and $0\le s \le L_\delta$ as well as $n \ge \chi_{\varepsilon,\alpha, \Psi}(\lambda)$ be given. We have, by Lemma \ref{lem:sumbound}, that
 \begin{equation*}
    \begin{aligned}
        &\sum_{m=n + 1}^\infty \PP\left(\left|\frac{S_{k_s^\pm(m)}}{k_s^\pm(m)} - \frac{z_{k_s^\pm(m)}}{k_s^\pm(m)}\right| > \varepsilon\right)\\
        &\qquad\qquad\le \frac{2^p\alpha^{2p}}{\varepsilon^p\floor{\alpha^{n+2}}^p(\alpha^p - 1)}\sum_{m =1}^{\floor{\alpha^{n+2}}} \gamma_m  +\frac{2^p\alpha^{2p}}{\varepsilon^p(\alpha^p - 1)}\sum_{m =\floor{\alpha^{n+1}}+1}^\infty \frac{\gamma_m}{m^p}.
    \end{aligned} 
    \end{equation*}
    Now, $n \ge \chi_{\varepsilon,\alpha, \Psi}(\lambda)$ implies
 $$n \ge \log_\alpha\left(2\Psi\left(\frac{\lambda\varepsilon^p(\alpha^p -1)}{2^{p+1}\alpha^{2p}}\right)\right)$$
 and from this, we deduce
 $$\floor{\alpha^{n+1}} \ge\alpha^{n+1}/2 \ge \Psi\left(\frac{\lambda\varepsilon^p(\alpha^p -1)}{2^{p+1}\alpha^{2p}}\right),$$ 
 which in turn implies
    \begin{equation*}
        \frac{2^p\alpha^{2p}}{\varepsilon^p(\alpha^p - 1)}\sum_{m =\floor{\alpha^{n+1}}+1}^\infty \frac{\gamma_m}{m^p}\le \frac{\lambda}{2}.
    \end{equation*}
    Now, observe that the partial sums of
    $$\frac{2^p\alpha^{2p}}{\varepsilon^p(\alpha^p - 1)}\sum_{m =1}^\infty \frac{\gamma_m}{m^p}$$ 
    converge, to their limit, with rate 
    $$\psi(\lambda) = \Psi\left(\frac{\lambda\varepsilon^p(\alpha^p - 1)}{2^p\alpha^{2p}}\right).$$
    Therefore, by Lemma \ref{lem:kronecker}, we have that
    $$\frac{2^p\alpha^{2p}}{\varepsilon^pn^p(\alpha^p - 1)}\sum_{m =1}^n \gamma_m$$ 
    converges to 0 with rate $K_{\psi,f_p,\seq{n^p}, R}$. Now, $n \ge \chi_{\varepsilon,\alpha, \Psi}(\lambda)$ further implies
    $$n \ge \log_\alpha\left(2K_{\psi,f_p,\seq{n^p}, R}\left(\frac{\lambda}{2}\right)\right)$$ 
    and, arguing as above, we get
    $$\floor{\alpha^{n+2}} \ge K_{\psi,f_p,\seq{n^p}, R}\left(\frac{\lambda}{2}\right).$$
    This allows us to conclude
       \begin{equation*}
   \frac{2^p\alpha^{2p}}{\varepsilon^p\floor{\alpha^{n+2}}^p(\alpha^p - 1)}\sum_{m =1}^{\floor{\alpha^{n+2}}} \gamma_m \le \frac{\lambda}{2}
    \end{equation*}
    and we are done.
 \end{proof}
 This, in particular, allows us rather immediately to deduce Theorem \ref{thrm:quantchen} and Theorem \ref{thrm:quant:cs}.
 \begin{proof}[Proof of Theorem \ref{thrm:quantchen}]
 In the context of the assumptions of Theorem \ref{thrm:quantchen}, the rate $\chi$ in the previous Lemma \ref{lem:chiConstr} simplifies to
\begin{equation*}
    \log_\alpha\left(\max\left\{2\Psi\left(\frac{\lambda\varepsilon^p(\alpha^p - 1)}{2^{p+3}\alpha^{2p}}\right),2\left\lceil\frac{2\alpha^2}{\varepsilon}\Psi\left(\frac{\lambda\varepsilon^p(\alpha^p - 1)}{2^{p+3}\alpha^{2p}}\right)\left(\frac{8\Gamma}{\lambda(\alpha^p -1)}\right)^\frac{1}{p}\right\rceil\right\}\right),
\end{equation*}
where we use the assumption that $\Psi$ is strictly decreasing to obtain such a simplification.

 We can now apply Theorem \ref{thrm:genquant} with the rate we obtained above, observing that $\chi$ is independent of $s$ (and $\delta$) to deduce Theorem \ref{thrm:quantchen}, noting that we can take $\alpha = 1 +\frac{\varepsilon}{3W}$ and so the assumption that $\varepsilon \le 1\le W$ implies $\alpha < 4/3$ and $\alpha^p -1 \ge \alpha -1$. Furthermore, the assumptions that $\Gamma, W \ge 1$ and $0<\varepsilon\le 1$, as well as the strictly decreasing condition on $\Psi$, allows us to conclude that the second argument in the $\max$ function above is bigger than the first argument. 
 \end{proof}
 \begin{proof}[Proof of Theorem \ref{thrm:quant:cs}]
      We write $X_n = X_n^+ - X_n^-$ and apply Theorem \ref{thrm:quantchen} to each sequence $\seq{X_n^\pm}$, with $p=2$ and $\gamma_n = \mathrm{Var}(X_n^\pm) \le \mathrm{Var}(X_n)$. We then obtain the result for $\seq{X_n}$ by arguing as in Proposition \ref{thrm:betterrate}.
 \end{proof}
\section{Comments and Questions}
\label{sec:conclusion}
There are still a lot of questions that can be asked in the area of large deviations in the Strong Law of Large Numbers, which we believe would be suited to the methods of proof mining. We conclude with some potential avenues of further study:

\begin{itemize}
    \item Firstly, as mentioned in Section \ref{subsec:related}, Cram\'{e}r's original 1938 results have been generalised in \cite{petrov122006large} to sequences of random variables not assumed to be identically distributed. However, the non-identically distributed generalisation of large deviation probabilities related to the weak law of large numbers \cite{bahadur1960deviations} has not been studied. So, we can ask if one can obtain such a generalisation. Furthermore, one can ask if such a generalisation exists for the Strong Law of Large Numbers.
   \item Baum-Katz type rates are typically given by showing some series converges through the bounding of the series. As already mentioned in this article, it is a well known result in computability theory that given a bound on a sum, there is no general computable process to extract a computable rate of convergence of that sum to its limit. Thus, one avenue of study is to try to extract rates of convergence for these Baum-Katz type results. Doing so will give us a more descriptive picture of how these large deviation probabilities behave. This problem appears to have already been considered in passing by Erd\H{o}s in \cite{Erdos:49:complete}. In the case $r = 0$ of Theorem \ref{thrm:baum-katzfull},  Erd\H{o}s provides an elementary proof that condition $(i)$ implies condition $(ii)$ (this was first demonstrated by Hsu and Robbins \cite{hsu1947complete} by techniques involving Fourier analysis) as well as the converse implication. Erd\H{o}s's approach to demonstrating that $(i)$ implies $(ii)$ was to split the sum $(ii)$ into three parts. For two parts, Erd\H{o}s calculates explicit rates of convergence that are independent of the distribution of the random variables; however, for the last part, Erd\H{o}s bounds the sum, and it is unclear how one obtains a rate from this bound (that is independent of the distribution of the random variables). Thus, getting a rate from Erd\H{o}s's proof is not straightforward. One could try to obtain a rate by analysing Hsu and Robbins's proof.
   \item Uniform rates of convergence (uniform in that they do not depend on the distribution of the random variables) have been found for the Central Limit Theorem: 
\begin{theorem}[Berry \cite{berry1941accuracy} and Esseen \cite{esseen1942liapunov}]
\label{thrm:BE}
    Let $\seq{X_n}$ be iid random variables satisfying $\EE(X_1) = 0, Var(X_1) = \sigma^2 >0, \EE(|X_1|^3) = \rho < \infty$. Let $$S_n = \frac{\sum_{i=1}^nX_i}{\sqrt{\sum_{i=1}^n\sigma_i^2}}$$ $F_n$ be the cumulative distribution function of $S_n$ and $\Phi$ the cumulative distribution function of the standard normal distribution.
    Then for all $n \in \NN$ and $x \in \RR$
\begin{equation*}
    |F_n(x)- \Phi(x)| \le \frac{C\rho}{\sqrt{n}\sigma^3}
\end{equation*}
    For some $C >0$.
\end{theorem}
 Suppose we do not assume the random variables have a finite third moment but do have finite variance. In that case, we can still deduce that $ F_n(x) \to \Phi(x)$ by the Central Limit Theorem, but we do not get such uniform rates of convergence (Berry initially attempted to do this and could not \cite{berry1941accuracy}). A similar phenomenon appears to occur in the Strong Law of Large Numbers, as the result holds if we assume the random variables have a finite first moment. It is also not clear if one can obtain rates that are independent of the distribution by only assuming a finite first moment. However, taking one moment higher (so finite variance) allows one to obtain rates independent of the distribution of the random variables. One can investigate whether this occurs in other limit theorems. A potential case study could investigate whether a rate of convergence for the Hsu-Robbins-Erd\H{o}s sum \cite{Erdos:49:complete,hsu1947complete} (condition $(ii)$ in Theorem \ref{thrm:baum-katzfull} in the case $r = 0$) can be obtained if we assume finite third moment instead of just finite variance. 
\end{itemize}

\noindent
{\bf Acknowledgments:}
This article was written as part of the author’s PhD studies under the supervision of Thomas Powell, and I would like to thank him for his support and invaluable guidance. The author would also like to thank Nicholas Pischke for his advice on the presentation and formatting of the results in this article, as well as C\'{e}cile Mailler and Nathan Creighton for their helpful comments.

The author was partially supported by the EPSRC Centre
for Doctoral Training in Digital Entertainment (EP/L016540/1).

\bibliographystyle{unsrt}
\bibliography{BIB}
\end{document}